\documentclass{article}

\usepackage{amsmath,
            amssymb,
            latexsym}

\usepackage[pdftex]{graphicx} 

\newtheorem{theorem}{Theorem}[section]
\newtheorem{lemma}[theorem]{Lemma}
\newtheorem{corollary}[theorem]{Corollary}

\newtheorem{proposition}[theorem]{Proposition}
\newtheorem{definition}[theorem]{Definition}
\newtheorem{example}[theorem]{Example}

\newenvironment{proof}{\noindent\textsc{Proof: }}
{\hspace{\stretch{1}}$\Box$\medskip}

\begin{document}

\title{Cut ideals of $K_4$-minor free graphs\\  are generated by quadrics.}

\author{
  Alexander Engstr\"om\footnote{The author is Miller Research Fellow 2009-2012 at UC Berkeley, and gratefully acknowledges support from the Adolph C. and Mary Sprague Miller Institute for Basic Research in Science.  \newline 
 The results of this paper were proved in the spring of 2008 when the author visited TU Berlin, and he would like to thank G\"unter Ziegler for his hospitality, and Bernd Sturmfels for proposing the problem.
  }
  \\
Department of Mathematics \\
UC Berkeley \\ 
\texttt{alex@math.berkeley.edu}
}

\date\today

\maketitle

\begin{abstract}
Cut ideals are used in algebraic statistics to study statistical models defined by graphs. Intuitively, topological restrictions on the graphs should imply structural statements about the corresponding cut ideals. Several theorems and many computer calculations support that.

Sturmfels and Sullivant conjectured that 
the cut ideal is generated by quadrics if and only if the graph
is free of $K_4$-minors. Parts of the conjecture has been resolved by Brennan and Chen, and later by Nagel and Petrovi\'c.

We prove the full conjecture by introducing a new type of toric fiber product theorem.
\end{abstract}

\section{Introduction}

In this paper we prove a conjecture by Sturmfels and Sullivant \cite{SS} about toric ideals used in algebraic statistics.
A new connection between commutative algebra and statistics was done by Diaconis and Sturmfels \cite{DiS} when they
introduced the fundamental notion of Markov basis. To explain the connection, we use the first example from the Oberwolfach lectures on algebraic statistics by Drton, Strumfels, and Sullivant \cite{DSS}.

\begin{example}\label{example:introduction}
\emph{In a contingency table, not only data is tabulated but also some marginals. In Table \ref{table:death} about death penalty verdicts the marginals are the row and column sums. To statistically test the hypothesis that the verdicts are from a distribution independent of race, one needs to sample from the set of tables with the same marginals as Table \ref{table:death}.}
\begin{table}
\begin{center}
\begin{tabular}{l|cc|c}
Defendant's race & Yes & No & Total \\
\hline
White & 19 & 141 & 160 \\
Black & 17 & 149 & 166 \\
\hline
Total & 36 & 290 & 326 \\
\end{tabular}
\end{center}
 \caption{Data on death penalty verdicts from \cite{A}, 5.2.2.}\label{table:death}
\end{table}
\emph{The usual way to sample is by a random walk on the set of tables with prescribed marginals, and stop when some test tells you that enough information is collected. The non-trivial task is to find good steps (\emph{Markov moves}) for the random walk, and here commutative algebra enters the picture.}

\emph{Encode the numbers in Table  \ref{table:death} with monomials as in Table \ref{table:monomials}. The data entries in Table \ref{table:monomials}
are collected in the monomial $q_{11}^{19}  q_{12}^{141}q_{21}^{17} q_{22}^{149} \in \mathbb{K}[q_{11},q_{12},q_{21},q_{22}]$
and the marginal entries in the monomial $r_{1\ast}^{160}  r_{2\ast}^{166} r_{\ast 1}^{36}  r_{\ast 2}^{290} \in \mathbb{K}[r_{1 \ast },r_{2 \ast },r_{\ast 1},r_{\ast 2}].$ The translation of calculating row and column sums into the algebraic setting is provided by the ring homomorphism
\[ \phi : \mathbb{K}[q_{11},q_{12},q_{21},q_{22}]  \rightarrow \mathbb{K}[r_{1 \ast },r_{2 \ast },r_{\ast 1},r_{\ast 2}]
\quad \textrm{ defined by } \quad \phi(q_{ab})=r_{a \ast }r_{\ast b}.  \]
The fiber of $r_{1\ast}^{160}  r_{2\ast}^{166} r_{\ast 1}^{36}  r_{\ast 2}^{290}$ are all monomials corresponding to tables with the same
marginals as in Table  \ref{table:death}, and that is the set of tables to find steps for. The kernel of the map $\phi$ is
a toric ideal, and a generating set of that ideal provides us with steps between the monomials in the fiber. In this easy example, the
kernel is generated by $q_{11}q_{22}-q_{12}q_{21}$, and for example provide a Markov move from 
$q_{11}^{19}  q_{12}^{141}q_{21}^{17} q_{22}^{149}$
to
$q_{11}^{20}  q_{12}^{140}q_{21}^{16} q_{22}^{150}$
since their difference is a monomial times $q_{11}q_{22}-q_{12}q_{21}$. 
All monomials in the fiber can be reached by Markov moves using $q_{11}q_{22}-q_{12}q_{21}$, and the statisticians are able to sample
from the set of tables with the same marginals as Table  \ref{table:death}.
   }
\begin{table}
\begin{center}
\[ \begin{array}{cc|c}
q_{11}^{19} & q_{12}^{141} & r_{1\ast}^{160} \\
q_{21}^{17} & q_{22}^{149} & r_{2\ast}^{166} \\
\hline
r_{\ast 1}^{36} & r_{\ast 2}^{290} & \\
\end{array}
 \]
\end{center}
 \caption{The commutative algebra version of Table \ref{table:death}.}\label{table:monomials}
\end{table}
\end{example}

The benefit of translating problems from statistics to commutative algebra as in Example~\ref{example:introduction} is the well developed tool-box for finding generators of ideals, most prominently using Gr\"obner basis.

Many statistical models are described by graphs, with a random variable for every vertex, and marginals described by edges. If we would flip a coin for every vertex in a graph, then the vertex set gets partitioned into two parts: heads and tails. A partition of a graph into two parts is called a \emph{cut} and many questions in statistics, computer science, and optimization theory are naturally formulated, or easily transformed into, questions about cuts. There is also a rich geometric theory associated to cuts, as surveyed by Deza and Laurent \cite{DL}.

\begin{definition}
For a graph $G$, the partition of $V(G)$ into $A$ and $B$ is the \emph{cut} $A\mid B = B \mid A$.
The edge set $ \{ ab \in E(G) \mid a \in A,\,  b\in B\}$ induced by the cut  $A \mid B$ is also denoted $A \mid B$ when no confusion arises. 
 \end{definition}

\begin{example}\label{ex:cut}
\emph{We toss four coins 76 times and get the statistic on eight different cuts in Table~\ref{table:numberOfCuts}. The marginals are encoded with the path graph $1-2-3-4$, and
in Table~\ref{table:cuts} are the cuts tabulated together with how they cut the edges. In Table~\ref{table:numberOfEdgeCuts} are the marginals of Table~\ref{table:numberOfCuts}
calculated, that is, how many times the different edges are cut. In algebraic statistic the corresponding setup is two commutative rings}
\[ \mathbb{K}\left[
\begin{array}{l}
q_{\{1,2,3,4\} \mid \emptyset},\,\,
q_{\{1,2,3\} \mid \{4\}},\,\,
q_{\{1,2,4\} \mid \{3\}},\,\,
q_{\{1,2\} \mid \{3,4\}},\\
q_{\{1,3,4\} \mid \{2\}},\,\,
q_{\{1,3\} \mid \{2,4\}},\,\,
q_{\{1,4\} \mid \{2,3\}},\,\,
q_{\{1\} \mid \{2,3,4\}} 
\end{array}
\right],
\]
\emph{and}
\[ \mathbb{K}[s_{12},s_{23},s_{34},t_{12},t_{23},t_{34}]; \]
\emph{and a ring homomorphism $\phi : \mathbb{K}[q] \rightarrow \mathbb{K}[s,t]$ defined by}
\[ 
\begin{array}{ll}
\phi(q_{\{1,2,3,4\} \mid \emptyset})=t_{12}t_{23}t_{34}, &
\phi(q_{\{1,2,3\} \mid \{4\}})=t_{12}t_{23}s_{34},\\
\phi(q_{\{1,2,4\} \mid \{3\}})=t_{12}s_{23}s_{34}, &
\phi(q_{\{1,2\} \mid \{3,4\}})=t_{12}s_{23}t_{34},\\
\phi(q_{\{1,3,4\} \mid \{2\}})=s_{12}s_{23}t_{34}, &
\phi(q_{\{1,3\} \mid \{2,4\}})=s_{12}s_{23}s_{34},\\
\phi(q_{\{1,4\} \mid \{2,3\}})=s_{12}t_{23}s_{34}, &
\phi(q_{\{1\} \mid \{2,3,4\}})=s_{12}t_{23}t_{34}, 
\end{array}
\]
\emph{in accordance with Table~\ref{table:cuts}, where $s_{ij}$ denotes that the edge $ij$ is separated and  $t_{ij}$ that it is kept together by the cut. The kernel of $\phi$ is a toric ideal generated
by the binomials}
\[
\begin{array}{c}
q_{\{1,3,4\} \mid \{2\}}q_{\{1,2,3\} \mid \{4\}} - q_{\{1,4\} \mid \{2,3\}}q_{\{1,2\} \mid \{3,4\}} \\
q_{\{1,3\} \mid \{2,4\}}q_{\{1,2,3,4\} \mid \emptyset} - q_{\{1\} \mid \{2,3,4\}}q_{\{1,2,4\} \mid \{3\}}, \\
q_{\{1,2,4\} \mid \{3\}}q_{\{1\} \mid \{2,3,4\}} - q_{\{1,4\} \mid \{2,3\}}q_{\{1,2\} \mid \{3,4\}}, \\
q_{\{1,3\} \mid \{2,4\}}q_{\{1,2,3,4\} \mid \emptyset} - q_{\{1,2,3\} \mid \{4\}}q_{\{1,3,4\} \mid \{1\}}.
\end{array}
\]
\begin{table}
\begin{center}
\[ \begin{array}{c|c}
\textrm{Cut} & \textrm{Number of occurrences}   \\
\hline
\{1,2,3,4\} \mid \emptyset & 8  \\
\{1,2,3\} \mid \{4\}  & 13  \\
\{1,2,4\} \mid \{3\} & 12  \\
\{1,2\} \mid \{3,4\}  & 6  \\
\{1,3,4\} \mid \{2\} & 9  \\
\{1,3\} \mid \{2,4\}  & 8  \\
\{1,4\} \mid \{2,3\} & 11  \\
\{1\} \mid \{2,3,4\}  & 9  \\
\end{array}
 \]
\end{center}
 \caption{The number of cuts of different types from tossing four coins 76 times.}\label{table:numberOfCuts}
\end{table}
\begin{table}
\begin{center}
\[ \begin{array}{c|ccc}
\textrm{Cut} & \textrm{Edge 12} & \textrm{Edge 23}  & \textrm{Edge 34}   \\
\hline
\{1,2,3,4\} \mid \emptyset & 0 & 0 & 0 \\
\{1,2,3\} \mid \{4\}  & 0 & 0 & 1 \\
\{1,2,4\} \mid \{3\} & 0 & 1 & 1 \\
\{1,2\} \mid \{3,4\}  & 0 &  1 & 0 \\
\{1,3,4\} \mid \{2\} & 1 & 1 & 0 \\
\{1,3\} \mid \{2,4\}  & 1 & 1 & 1 \\
\{1,4\} \mid \{2,3\} & 1 & 0 & 1 \\
\{1\} \mid \{2,3,4\}  & 1 & 0 & 0 \\
\end{array}
 \]
\end{center}
 \caption{The cuts of the path graph $1-2-3-4$. Edges with vertices in different parts are tabulated with 1 and those in the same parts with 0.}\label{table:cuts}
\end{table}
\begin{table}
\begin{center}
\[ \begin{array}{c|ccc}
  & \textrm{Edge 12} & \textrm{Edge 23}  & \textrm{Edge 34}   \\
\hline
\# \textrm{ Cuts} & 37 & 35 & 44 \\
\end{array}
 \]
\end{center}
 \caption{The edge cuts of the path graph $1-2-3-4$ given
 the cut statistic in Table~\ref{table:numberOfCuts}.}\label{table:numberOfEdgeCuts}
\end{table}
\end{example}

The toric ideal in Example~\ref{ex:cut} is the cut ideal of a path on four vertices.
The theory of cut ideals was initiated by Sturmfels and Sullivant \cite{SS}.
\begin{definition}
The \emph{cut ideal of the graph $G$}, $I_G$, is the kernel of the ring homomorphism
$\phi_G : \mathbb{K}[q]\rightarrow \mathbb{K}[s,t]$ defined by
\[  q_{A\mid B}\mapsto \prod_{ij {\tiny \textrm{ is in } A\mid B}} s_{ij} \prod_{ij {\tiny \textrm{ is not in } A\mid B}} t_{ij}, \]
where
\[ \mathbb{K}[q]  =  \mathbb{K}[q_{A \mid B}\, \mid \, \textrm{there is a cut $A | B$ of $G$}], \]
\[ \mathbb{K}[s,t] =  \mathbb{K}[s_{ij},t_{ij} \, \mid \, \textrm{$ij$ is an edge of $G$}]. \]
\end{definition}
The applications of cut ideals in statistics and in the applied sciences are not apparent from Example~\ref{ex:cut}, since it's too small. As described in \cite{SS} there are applications in biology \cite{PS}, for example by the Jukes-Cantor model.

From theorems about similar constructions, and computer calculations, it is reasonable to believe that topological properties of $G$ should be reflected in algebraic properties of $I_G$.\newline
\newline
{\bf Theorem (Conjectured by Sturmfels and Sullivant \cite{SS})} \emph{  The cut ideal is generated by quadrics if and only if $G$ is free of  $K_4$ minors.}\newline

Several partial results have been proved: Brennan and Chen \cite{BC} showed it for subdivisions of books and outerplanar graphs. A ring graph is, more or less, a bunch of disjoint
cycles that are connected by a tree that touch any cycle in at most one vertex. For ring graphs the conjecture was proved by Nagel and Petrovi\'c \cite{NP}.

The conjecture follows as a corollary of Theorem~\ref{theo:mulimit}, which
is a fiber product type theorem. In the same way as the fiber
product theorems in \cite{DS} and \cite{SS} could be generalized
in \cite{S}, we will present a more general form of 
Theorem~\ref{theo:mulimit} in \cite{E}. 
Methods from this paper were used on ideals of graph homomorphisms in Engstr\"om and Nor\'en's paper \cite{EN}.

\subsection{Basic notions of cut ideals}

The largest degree of a minimal generator of $I_G$ is $\mu(G)$.
By Corollary 3.3 of \cite{SS} the contraction of an edge or deletion of a vertex cannot increase $\mu$.
In Theorem 2.1 of \cite{SS} it is proved that if $G$ is glued together from two graphs $G_1$ and $G_2$ over a
complete graph with zero, one, or two vertices, then the cut ideal $I_G$ is generated by lifts
of generators of $I_{G_1}$ and $I_{G_2}$; and quadratic binomials for sorting cuts.
The main theorem of this paper is a variation on Theorem 2.1 of \cite{SS} when gluing over an edge.  

\section{Decompositions of graphs and ideals}

The induced subgraph of $G$ on $S$ is denoted $G[S]$. 

\begin{definition} Let $u,v$ be two vertices of $G$ and
   $A_1 \mid B_1, A_2\mid B_2,\cdots,A_n\mid B_n$ a list of
   cuts. The \emph{height}, $h_{u,v}(q)$, of
   \[q=q_{A_1\mid B_1}q_{A_2\mid B_2}\cdots q_{A_n\mid B_n}\]
   with respect to $u$ and $v$ is the number of cuts in
   the list with $u$ and $v$ in different parts.
\end{definition}

If there is an edge between $u$ and $v$ in $G$ then $h_{u,v}(q)$ is the
degree of $s_{uv}$ in $\phi_G(q)$. Another way to define the
height of $q$ with respect to $u$ and $v$ is as the degree of
$s_{uv}$ in $\phi_{G+uv}(q)$, and that is a good way to think of it.

\begin{definition}
   A set of generators
   \[ q_{A_{i,1}\mid B_{i,1}}q_{A_{i,2}\mid B_{i,2}}\cdots q_{A_{i,n_i}\mid B_{i,n_i}} 
      -  q_{A_{i,1}'\mid B_{i,1}'}q_{A_{i,2}'\mid B_{i,2}'}\cdots q_{A_{i,n_i}'\mid B_{i,n_i}'}  \]
   of $I_G$ is \emph{slow-varying} with respect to the vertices $u$ and $v$ of $G$ if
   \[  \left| h_{u,v}(q_{A_{i,1}\mid B_{i,1}}\cdots q_{A_{i,n_i}\mid B_{i,n_i}})
     -h_{u,v}(q_{A_{i,1}'\mid B_{i,1}'}\cdots q_{A_{i,n_i}'\mid B_{i,n_i}'}) \right| \leq 2   \]
   for all $i$.
\end{definition}

\begin{lemma}\label{lemma:walkmod}
   If $w_1-w_2-\cdots-w_k$ is a path in $G$ then
  \[h_{w_1,w_k}(q_{A\mid B}) \equiv \sum_{i=1}^{k-1} ( s_{w_iw_{i+1}}-\textrm{degree of } \phi_G(q_{A\mid B}))\]
modulo 2.
\end{lemma}
\begin{proof}
A walk on the path from $w_1$ to $w_k$ crosses the cut an odd number of times if and only if $w_1$ and $w_k$ are in different parts.
\end{proof}

\begin{lemma}\label{lemma:evenpath}
If there is a path in $G$ from $u$ to $v$ and $\phi_G(q)=\phi_G(q')$ then $h_{u,v}(q)\equiv h_{u,v}(q')$ modulo 2.
\end{lemma}
\begin{proof}
Use Lemma~\ref{lemma:walkmod}.
\end{proof}

\begin{proposition}\label{prop:2vary}
   Any set of generators of $I_G$ validating $\mu(G)\leq 2$, is
   slow-varying with respect to any vertex pair.
\end{proposition}

\begin{proof} Clear. \end{proof}

\begin{theorem}\label{theo:mulimit}
  Let $G$ be a graph with two special non-adjacent vertices $u$ and $v$.
  Assume that $G$ almost can be decomposed into a left and right part:
  There are $L,R\subseteq V(G)$ such that $L\cup R=V(G)$, $L\cap R=\{u,v\}$,
  and $E(G)=E(G[L])\cup E(G[R])$.

  If there is a path from $u$ to $v$ both in $G[L]$ and in $G[R]$, and there
  are slow-varying generators of both $I_{G[L]}$ and $I_{G[R]}$ with respect to $u$ and $v$, then
  \[ \mu(G) \leq \max \{ 2\mu(G[L])-2, 2\mu(G[R])-2, \mu(G[L]+uv), \mu(G[R]+uv) \}. \]
  
  The cut ideal of $G$ is generated by a union of
  \begin{itemize}
     \item[(i)] lifts of generators of $I_{G[L]+uv}$,
     \item[(ii)] lifts of generators of $I_{G[R]+uv}$,
     \item[(iii)] joins of generators $q_{1}-q_{2}$ of $I_{G[L]}$
                  and $q_{3}-q_{4}$ of $I_{G[R]}$ such that
                  $|h_{u,v}(q_1)-h_{u,v}(q_2)|=|h_{u,v}(q_3)-h_{u,v}(q_4)|=2$,
     \item[(iv)] quadratic binomials to reorder with.
  \end{itemize}
\end{theorem}
  
\begin{proof}
The basic part of this proof, only involving $G[L]$ and $G[R]$, is in the spirit of the proof of Theorem 2.1 in \cite{SS}.

  We will prove the theorem by an explicit construction of generators
  for $I_G$.
  
  Let
  \[ q = \prod_{i=1}^n q_{A_i \mid B_i}
         \quad \quad \textrm{and} \quad \quad
     q' = \prod_{i=1}^n q_{A'_i \mid B'_i}
  \]
  be two elements of $\mathbb{K}[q_{A\mid B} \mid A \sqcup B = V(G) ]$ with
  $\phi_G(q)=\phi_G(q')$. If we for any such $q$ and $q'$ can construct a sequence
  of moves from $q$ to $q'$, then we can generate $I_G$. A move from $q_1$ to
  $q_2$ is a composition of a $q_3$ with a binomial generator $q_4-q_5$ such that
  \[ q_1-q_2 = q_3(q_4-q_5). \]
  We can assume that
  $h_{u,v}(q)\geq h_{u,v}(q')$ 
  
  {\bf Main idea:} To construct the sequence from $q$ to $q'$ we use sequences
  from $q_L$ to $q_L'$ and from $q_R$ to $q_R'$. ($q_L$ is $q$ induced on $L$ and 
  similar for $q_R$.) If we simply took a sequence from $q_L$ to $q_L'$ given by $I_{G[L]}$ and a 
  corresponding one on $R$ and tried to glue them together it would sometimes not work on
  the vertex pair $u$ and $v$. The thing that goes wrong is that the number of
  cuts with $u$ and $v$ in different parts does not need to be the same. That is,
  the height $h_{u,v}$ could be different on the left and the right side.
  But we know that the height is the same 
  for $q_L$ and $q_R$ in the begining of the sequence, and for $q_L'$ and $q_R'$
  in the end of the sequence.

  In the sequence $q_L, \ldots, q'_L$ the number of cuts with $u$ and $v$ in 
  different parts can look like the fat gray line in Figure \ref{fig:movingAround}.
\begin{figure}
  \begin{center}
  \includegraphics[scale=0.75]{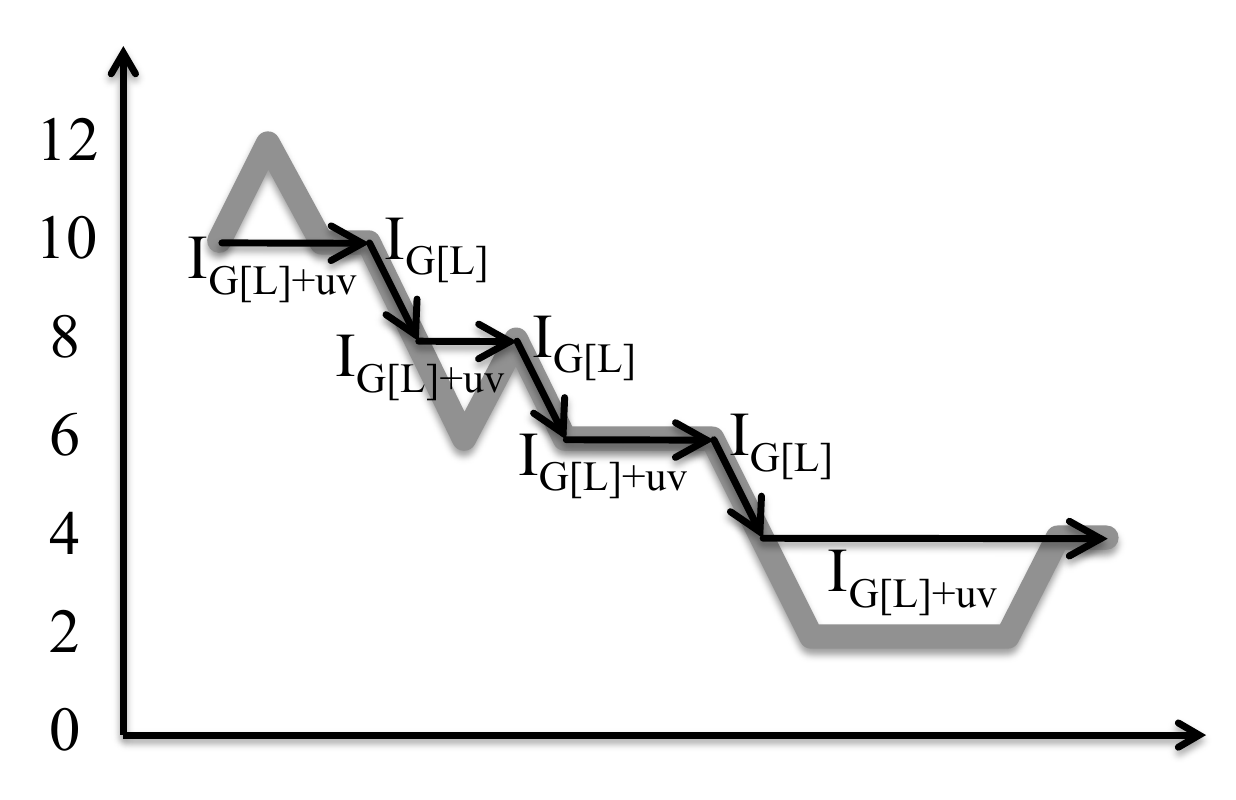}
  \end{center}
  \caption{On the vertical axis is the height with respect to $u$,$v$. }
  \label{fig:movingAround}
\end{figure}
  If it changes, it changes by an even number by Lemma~\ref{lemma:walkmod}.
  It never changes by more than 2 since
  $I_{G[L]}$ is slow-varying. Since the height of the sequence $q_R, \ldots, q'_R$ does not
  have to have the same shape as the grey line, we need to normalize the sequences.
  
  {\bf How to normalize the sequence $q_L, \ldots, q'_L$:} We do this as described
  in Figure \ref{fig:movingAround}. Let $q'_{L,h}$ be the last element in the
  sequence with height $h$ for $h=h_{u,v}(q_L), h_{u,v}(q_L)-2,\ldots, h_{u,v}(q'_L)+2, h_{u,v}(q'_L)$.
  Let $q_{L,h}$ be the element after $q'_{L,h+2}$ in the sequence for 
  $h=h_{u,v}(q_L)-2,\ldots, h_{u,v}(q'_L)+2, h_{u,v}(q'_L)$. And let 
  $q_{L,h_{u,v}(q_L)}=q_L$. In our normalized sequence we still go from
  $q'_{L,h}$ to $q_{L,h-2}$ by using a generator of $I_{G[L]}$. But from
  $q_{L,h}$ to $q'_{L,h}$ we build up the sequence by using generators of
  $I_{G[L]+uv}$, this is possible since the heights of $q_{L,h}$ and
  $q'_{L,h}$ are the same. For our normalized sequence the height is
  never increasing.

  Normalize $q_R, \ldots, q'_R$ the same way. The plot of the heights
  for the normalized sequences on $L$ and $R$ now looks the same and
  we can put the sequences together without any conflicts on $u$ and $v$. 

  Thus we need four kinds of moves:
  \begin{itemize}
     \item[($\mathbf{F}_1$)] all from $I_{G[L]+uv}$, 
     \item[($\mathbf{F}_2$)] all from $I_{G[R]+uv}$,
     \item[($\mathbf{F}_3$)] those from $I_{G[L]}$ and $I_{G[R]}$ that change height by 2,
     \item[($\mathbf{F}_4$)] reorderings to match cuts.
  \end{itemize} 

  Let $\mathbf{F}_L,$ $\mathbf{F}_{L+uv},$ $\mathbf{F}_R,$ and $\mathbf{F}_{R+uv}$
  be the binomial generating sets of $I_{G[L]},$ $I_{G[L]+uv},$ $I_{G[R]},$ and 
  $I_{G[R]+uv}$. If the maximal degree of a binomial in $\mathbf{F}_L$ or $\mathbf{F}_R$
  is $M$ then extend $\mathbf{F}_L$ to
  \[ \tilde{\mathbf{F}}_L = \{ q_1(q_2-q_3) \mid \textrm{degree of }q_1q_2\leq 2M-2 \textrm{ and }
                               q_2-q_3 \in \mathbf{F}_L \} \]
  and $\mathbf{F}_R$ to
  \[ \tilde{\mathbf{F}}_R = \{ q_1(q_2-q_3) \mid \textrm{degree of }q_1q_2\leq 2M-2 \textrm{ and }
                               q_2-q_3 \in \mathbf{F}_R \}. \]
  The extension is needed to allow binomial generators of different degree from the left and
  right side to be joined when the height decreases by two.
  In the definitions of $\mathbf{F}_1$,$\mathbf{F}_2$, and $\mathbf{F}_3$, any product of the type
  \[ \prod_{i=1}^m q_{C_i \mid D_i}\]
  is assumed to have an order such that
  \[ h_{u,v}(q_{C_1 \mid D_1})\geq \cdots \geq  h_{u,v}(q_{C_m \mid D_m}). \]
  Let
  \[ \mathbf{F}_1 = \left\{ \prod_{i=1}^m q_{C_i \mid D_i} - \prod_{i=1}^m q_{C_i' \mid D_i'}
                    \in \mathbb{K}[q_G] \left| \begin{array}{l}
                    \prod_{i=1}^m q_{C_i\cap L \mid D_i\cap L} - \prod_{i=1}^m q_{C_i'\cap L \mid D_i'\cap L} \in \mathbf{F}_{L+uv} \\
                    C_i\cap R = C_i'\cap R \textrm{ for }i=1,\ldots m \\ 
                    \end{array}
                    \right.               
                    \right\} \]
  \[ \mathbf{F}_2 = \left\{ \prod_{i=1}^m q_{C_i \mid D_i} - \prod_{i=1}^m q_{C_i' \mid D_i'}
                    \in \mathbb{K}[q_G] \left| \begin{array}{l}
                    \prod_{i=1}^m q_{C_i\cap R \mid D_i\cap R} - \prod_{i=1}^m q_{C_i'\cap R \mid D_i'\cap R} \in \mathbf{F}_{R+uv} \\
                    C_i\cap L = C_i' \cap L \textrm{ for }i=1,\ldots m \\ 
                    \end{array}
                    \right.               
                    \right\} \]
  \[ \mathbf{F}_3 = \left\{ \prod_{i=1}^m q_{C_i \mid D_i} - \prod_{i=1}^m q_{C_i' \mid D_i'}
                    \in \mathbb{K}[q_G] \left| \begin{array}{l}
                    \prod_{i=1}^m q_{C_i\cap L \mid D_i\cap L} - \prod_{i=1}^m q_{C_i'\cap L \mid D_i'\cap L} \in \tilde{\mathbf{F}}_{L} \\
                    \prod_{i=1}^m q_{C_i\cap R \mid D_i\cap R} - \prod_{i=1}^m q_{C_i'\cap R \mid D_i'\cap R} \in \tilde{\mathbf{F}}_{R} \\
                    h_{u,v}\left(  \prod_{i=1}^m q_{C_i \mid D_i} \right) \neq h_{u,v}\left(  \prod_{i=1}^m q_{C_i \mid D_i} \right)
                    \end{array}
                    \right.               
                    \right\} \]
  \[ \mathbf{F}_4 = \left\{ \prod_{i=1}^2 q_{C_i \mid D_i} - \prod_{i=1}^2 q_{C_i' \mid D_i'}
                    \in \mathbb{K}[q_G] \left| \begin{array}{ll}
                    C_1\cap L = C_1'\cap L, & C_2\cap L  = C_2'\cap L \\
                    C_1\cap R = C_2'\cap R, & C_2\cap R  = C_1'\cap R \\
                    \end{array}
                    \right.               
                    \right\}. \]
We have that $\mathbf{F}=\mathbf{F}_1 \cup \mathbf{F}_2 \cup \mathbf{F}_3 \cup \mathbf{F}_4$ is a generating set of $I_G$.
From that we get:
\[ \mu(G) \leq \max \{ 2, 2\mu(G[L])-2, 2\mu(G[R])-2, \mu(G[L]+uv), \mu(G[R]+uv) \} \]
In $G[L]$ there is an induced path from $u$ to $v$ with more than one edge. For the path with two edges we have
$\mu =2$ and thus by contraction $\mu\geq 2$ for any path, which shows that $\mu(G[L])\geq 2$. The 2 can be removed to get:
\[ \mu(G) \leq \max \{ 2\mu(G[L])-2, 2\mu(G[R])-2, \mu(G[L]+uv), \mu(G[R]+uv) \} \]
\end{proof}

\begin{corollary}\label{cor:main}
  Let $H_1$ and $H_2$ be two graphs on different vertex sets satisfying:
\begin{itemize}
  \item $u_1,v_1$ are two distinct non-adjacent vertices of $H_1$,
  \item $u_2,v_2$ are two distinct non-adjacent vertices of $H_2$,
  \item $H_1$ and $H_2$ are connected,
  \item $\mu(H_1), \mu(H_2), \mu(H_1+u_1v_1), \mu(H_2+u_2v_2)\leq 2.$
\end{itemize}
  Then $\mu\leq 2$ for the graph we get by gluing $u_1=u_2$
  and $v_1=v_2$ in $H_1\cup H_2$.
\end{corollary}

\begin{proof}
  Insert Proposition \ref{prop:2vary} into Theorem \ref{theo:mulimit}.
\end{proof}

The graphs without $K_4$-minors are also called series-parallell graphs. Starting with the complete graphs on less than four vertices, the connected
series-parallel graphs can be constructed by the gluing two smaller ones in series over one vertex, or in parallel over two vertices that could be connected or not \cite{D}.

\begin{corollary}[Conjecture 3.5 of Sturmfels and Sullivant \cite{SS}]
  The cut ideal is generated by quadrics if and only if $G$ is free of
  $K_4$-minors.
\end{corollary}

\begin{proof}
  We prove that if $G$ is series-parallel then $\mu(G)\leq 2$.
  The other direction was proved in \cite{SS}.
  We only need to prove it for connected series-parallel graphs.

  The proof is by induction on the number of vertices of $G$. If
  there are less than four vertices then $\mu(G) \leq 2$ by
  explicit calculations in \cite{SS}.

  Now assume that $G$ has at least four vertices. If $G$ is
  constructed by two graphs $H_1$ and $H_2$ put in series and
  glued at one vertex, then $\mu(G)=\max\{\mu(H_1),\mu(H_2)\}\leq 2$
  by the fiber construction in \cite{SS}.

  If $G$ is constructed by two graphs $H_1$ and $H_2$ glued 
  parallel together in two vertices we have two cases.

  The first case: However subgraphs $H_1$ and $H_2$ are choosen to
  be glued together in parallel to create $G$, 
  one of them will only be an edge. 

  Assume that $H_2$ is only the
  edge $uv$, and that $uv$ is not in $H_1$. If $H_1$ came from a
  parallel gluing of $H_1'$ and $H_1''$ at $u$ and $v$, then $G$
  could be parallel constructed from $H_1'$ and $H_1''+uv$ and none
  of them is only an edge, which is a contradiction. So $H_1$ is
  from a series gluing at some vertex $w\not\in \{u,v\}$. Both
  graphs glued together to get $H_1$ cannot be only edges, since then
  $G$ is a triangle, and we assumed $G$ to have more than 3 vertices.
  Thus we can assume that the part of $H_1$ between $v$ and $w$ have more
  than two vertices. But then $G$ can be formed as a parallel
  construction glued at $v$ and $w$ where none of the parts is only
  an edge, and that situation is the second case.

  The second case: The graph $G$ can be created by a parallel
  construction at $u,v$ of two graphs $H_1$ and $H_2$ and both
  of them have more than two vertices. If $uv$ is an edge of $G$
  then $\mu(G)=\max \{ \mu(H_1+uv), \mu(H_2+uv) \}\leq 2$
  since $H_1$ and $H_2$ are series-parallel. If there is no
  edge between $u$ and $v$ in $G$ we use that $H_1$, $H_2$,
  $H_1+uv$, $H_2+uv$ are series-parallel and Corollary~\ref{cor:main}
  to get that $\mu(G)\leq 2$.
  
\end{proof}

\end{document}